\documentclass[12pt,a4paper]{article}

\usepackage{graphicx,enumitem}

\usepackage{amsmath,amsthm,amssymb}
\usepackage{color,tikz}
\usepackage[noadjust,sort]{cite}
\usepackage{hyperref}

\usepackage[font=small,labelfont=bf]{caption}
\normalsize

\newtheorem{thm}{Theorem}[section]
\newtheorem{cor}[thm]{Corollary}
\newtheorem{lemma}[thm]{Lemma}

\numberwithin{equation}{section}

\newcommand{\dmax}{d_{\mathrm{max}}}

\newcommand{\smax}{s_{\mathrm{max}}}
\newcommand{\tmax}{t_{\mathrm{max}}}
\newcommand{\xmax}{x_{\mathrm{max}}}
\newcommand{\ymax}{y_{\mathrm{max}}}
\newcommand{\deltamax}{\delta_{\mathrm{max}}}

\let\le=\leqslant
\let\ge=\geqslant
\let\leq=\leqslant
\let\geq=\geqslant

\let\originalleft\left
\let\originalright\right
\renewcommand{\left}{\mathopen{}\mathclose\bgroup\originalleft}
\renewcommand{\right}{\aftergroup\egroup\originalright}

\newcommand{\abs}[1]{\lvert#1\rvert} \let\card=\abs
\newcommand{\Abs}[1]{\bigl\lvert#1\bigr\rvert} 

\renewcommand{\dfrac}[2]{\lower0.12ex\hbox{\large$\textstyle\frac{#1}{#2}$}}
\newcommand{\Dfrac}[2]{\raise0.05ex\hbox{\small$\displaystyle\frac{#1}{#2}$}}
\newcommand{\eps}{\varepsilon}

\newcommand{\calB}{\mathcal{B}}
\newcommand{\calT}{\mathcal{T}}
\newcommand{\dvec}{\boldsymbol{d}}
\newcommand{\deltavec}{\boldsymbol{\delta}}
\newcommand{\svec}{\boldsymbol{s}}
\newcommand{\tvec}{\boldsymbol{t}}
\newcommand{\jvec}{\boldsymbol{j}}
\newcommand{\xvec}{\boldsymbol{x}}
\newcommand{\yvec}{\boldsymbol{y}}
\newcommand{\wvec}{\boldsymbol{w}}
\newcommand{\zvec}{\boldsymbol{z}}
\newcommand{\X}{\boldsymbol{X}}

\newcommand{\tc}[2]{#1{\cdot}#2}

\newcommand{\Reals}{\mathbb{R}}
\newcommand{\E}{\operatorname{\mathbb{E}}}

\newcommand{\EO}{\operatorname{EO}}
\newcommand{\Var}{\operatorname{Var}}

\def\nicebreak{\vskip 0pt plus 50pt\penalty-300\vskip 0pt plus -50pt }

\begin{document}

\title{Asymptotic enumeration of constrained bipartite,
directed and oriented graphs by degree sequence}
\author{
Catherine Greenhill\thanks{Supported by Australian Research Council grant DP250101611.}\\
\small School of Mathematics and Statistics\\[-0.8ex]
\small UNSW Sydney\\[-0.8ex]
\small NSW 2052, Australia\\
\small \tt c.greenhill@unsw.edu.au
\and
Mahdieh Hasheminezhad\thanks{Supported by a MATRIX-Simons travel grant.}\\
\small Department of Computer Science\\[-0.5ex]
\small  Yazd  University\\ [-0.5ex]
\small  Yazd, Iran\\[-0.5ex]
\small\tt hasheminezhad@yazd.ac.ir
\and
Isaiah Iliffe\\
\small School of Mathematics and Statistics\\[-0.8ex]
\small UNSW Sydney\\[-0.8ex]
\small NSW 2052, Australia\\
\small \tt isaiahiliffe@gmail.com
\and
Brendan D. McKay\footnotemark[1] \\
\small School of Computing \\[-0.5ex]
\small Australian National University  \\[-0.5ex]
\small Canberra ACT 2601, Australia \\[-0.5ex]
\small\tt brendan.mckay@anu.edu.au
}

\date{}

\maketitle
\begin{abstract}
In the sufficiently sparse case, we find the probability that 
a uniformly random bipartite graph with given degree sequence
contains no edge from a specified set of edges.
This enables us to enumerate loop-free digraphs and
oriented graphs with given in-degree and out-degree
sequences, and obtain subgraph probabilities.
Our theorems are not restricted to the near-regular case.
As an application, we determine the expected permanent of sparse or
very dense random matrices with given row and column sums;
 in the regular case, our formula holds over all densities.
We also draw conclusions about the degrees  of a random
orientation  of a random undirected graph with given
degrees, including its number of Eulerian orientations.
\end{abstract}

\section{Introduction}

A graph is bipartite if we can partition its vertex set into two disjoint nonempty sets,
say $U$ and $V$, such that all edges contain a vertex from $U$ and a vertex from $V$.
All graphs in this paper are finite. We will focus on bipartite graphs with a given
vertex bipartition $U\cup V$, say $U = \{u_1,\ldots,u_m\}$
and $V = \{ v_1,\ldots, v_n\}$. Given a pair of vectors $(\svec,\tvec)$ of nonnegative
integers, $\svec = (s_1,\ldots, s_m)$, $\tvec = (t_1, \ldots, t_n)$, we say that $(\svec,\tvec)$
is the degree sequence of a given bipartite graph on $U\cup V$ if, for all $i,j$, $u_i$ has degree $s_i$ and $v_j$ has degree $t_j$.
Our first goal is to present
an asymptotic formula for the number of bipartite graphs with given degree sequence which
avoid all edges of a  specified graph $X$, under certain sparseness conditions on the degree sequence and on $X$.
This result is Theorem~\ref{thm:bip-avoid}.

An ordered pair $G=(V,E)$ is a \textit{directed graph} (or \textit{digraph}) if $V$ is a finite and nonempty set and $E$ is a subset of $V \times V$. 
The members of $V$ are called \textit{vertices} of $G$ and the members of $E$ are called \textit{edges} of $G$.  An edge $(v,u) \in E$ is an \textit{outgoing edge} from vertex $v$ and an \textit{incoming edge} to vertex $u$. 
The number of outgoing edges from a vertex $v \in V$ is called the \textit{out-degree} of $v$ and the number of incoming edges to a vertex $v \in V$ is called the \textit{in-degree} of~$v$.
Let $G$ be a directed graph on the vertex set $W = \{ w_1,\ldots, w_n\}$ with out-degree sequence $\svec$ and in-degree sequence $\tvec$. We will say that the directed graph $G$ has degree sequence $(\svec,\tvec)$. Since a loop-free directed graph
can be modelled as a bipartite graph which contains no
edge of a specified perfect matching, 
we obtain from our
first result an asymptotic enumeration formula for loop-free directed graphs with given out-degree and in-degree sequences, again under a sparsity condition.  See Corollary~\ref{loopfreethm}. Another application of our first
result provides a formula for the expected permanent of sparse random matrices with given row and column sums, see Theorem~\ref{permanent}. With some more work we find a formula for the expected permanent which holds over all densities when all row and column sums are equal. This result is Theorem~\ref{regulardet}.

An \textit{oriented graph} is a digraph which contains neither loops nor directed 2-cycles. Every oriented graph can be obtained from a simple undirected graph by orienting its edges. We obtain a formula for the number of oriented graphs with given degrees in Corollary~\ref{2cycans}, under a sparsity condition. Finally, in Theorem~\ref{orients} we give an asymptotic formula for the expected number of ways to orient a random undirected graph with a given (sparse) degree sequence, such that the resulting orientations have specified in-degrees and out-degrees.  In particular, this gives the expected number of Eulerian orientations of a random graph with given (sparse) degrees, when all of these degrees are even. 

Results on bipartite graphs are presented in Section~\ref{s:bipartite}, then we consider digraphs in Section~\ref{s:digraphs} and permanents in Section~\ref{s:permanents}. Finally in Section~\ref{s:oriented} we consider oriented  graphs.

Before proceeding we make a couple of quick remarks about notation.
For a positive integer $a$, let $[a]:=\{1,2,\ldots, a\}$. We write $(x)_b=x(x-1)\cdots(x-b+1)$ for the falling factorial, where $x$ is a real number and $b$ is a nonnegative integer.  We will identify bipartite graphs and directed graphs with their edge sets.

\section{Bipartite graphs}\label{s:bipartite}

We consider bipartite graphs with vertex bipartition $U\cup V$, where $U=\{ u_1,\ldots, u_m\}$ and $V= \{ v_1,\ldots, v_n\}$.
Let $\mathcal{B}(\svec,\tvec)$ denote the set of 
simple bipartite graphs with degree sequence $(\svec,\tvec)$,
where
$\svec=(s_1,\ldots,s_m)$ and $\tvec=(t_1,\ldots,t_n)$.
That is, vertex $u_j$ has degree $s_j$ for all $j\in [m]$ and
vertex $v_j$ has degree $t_j$ for all $j\in [n]$.
Write $B(\svec,\tvec)=|\mathcal{B}(\svec,\tvec)|$.

Define, for all nonnegative integers $b$,
\begin{gather*}
   \smax = \max_{i\in [m]} s_i, \quad
   \tmax = \max_{j\in [n]} t_j, \quad
   \quad 
    S = \sum_{i\in [m]} s_i = \sum_{j\in [n]} t_j; \\
   S_b=\sum_{i\in [m]} (s_i)_b,\quad
     T_b=\sum_{j\in [n]} (t_j)_b.
\end{gather*}
Elementary bounds apply, such as $S_2\le\smax S$ 
and $T_2\le\tmax S$.

In this section we will count bipartite graphs with a given degree sequence
which contain no edge of a specified bipartite graph. Our starting point is the following result from~\cite{GMW}.

\begin{thm}[{\cite[Theorem\ 1.3]{GMW}}]\label{GMW}
If $S\to\infty$ and $\smax\tmax=o(S^{2/3})$,
then the number of bipartite graphs with degrees
$\svec,\tvec$ is
\[
B(\svec,\tvec) = \frac{S!}
 {\prod_{i\in [m]} s_i!\;\prod_{j\in [n]} t_j!}
\exp\biggl( Q(\svec,\tvec)
+ O\biggl(\frac{\smax^3\tmax^3}{S^2}\biggr)\biggr),
\]
where
\[
  Q(\svec,\tvec) = -\frac{S_2T_2}{2S^2} - \frac{S_2T_2}{2S^3} 
+ \frac{S_3T_3}{3S^3} 
 - \frac{S_2T_2(S_2+T_2)}{4S^4} 
 - \frac{S_2^2T_3+S_3T_2^2}{2S^4}
+ \frac{S_2^2T_2^2}{2S^5} .
\]
\end{thm}

\bigskip

Let $X\subseteq U\times V$ specify (the edge set of) a
bipartite graph on $U\cup V$, and let $B(\svec,\tvec,X)$ be
the number of graphs in $\mathcal{B}(\svec,\tvec)$ which contain no edge of the graph $X$.
Define the parameter
\[ F = F(X):= \sum_{u_iv_j\in X} s_it_j\]
and let $(\xvec,\yvec)$ be the degree sequence of $X$. 
That is, vertex $u_i$ is contained in exactly $x_i$ edges of $X$ for all $i\in [m]$, and vertex $v_j$ is contained 
in exactly $y_j$ edges of $X$ for all $j\in [n]$.  Finally, let
\[  \xmax = \max_{j\in [m]} x_j,\quad \ymax = \max_{j\in [n]} y_j, \quad \deltamax =\smax\tmax+\smax\ymax+\xmax\tmax.\]

McKay~\cite[Theorem~4.6]{silver} gave an asymptotic formula for $B(\svec,\tvec,X)$ which is precise when $O\bigl( (\smax + \tmax)(\smax + \tmax + \xmax + \ymax)\bigr)= o(S^{1/2})$.  For very dense degrees, Greenhill and McKay~\cite[Theorem~2.1]{GMbip} provided an asymptotic enumeration formula for $B(\svec,\tvec,X)$ which allows $|X|$ to be slightly superlinear in $n$.
Liebenau and Wormald~\cite{LW2022} gave a formula for $B(\svec,\tvec)$ which holds for near-regular degree sequences of a large range of densities.

\bigskip

We now state the main result of this section, which extends McKay~\cite[Theorem~4.6]{silver}.

\begin{thm}
\label{thm:bip-avoid}
Let $X\subseteq U\times V$ be a specified bipartite graph on $U\cup V$.
Suppose that $\smax +\tmax = o(S/\log S)$, $\deltamax = o(S)$,
$\deltamax F = o(S^2)$ and $F =o(S^{5/3})$. 
Then
\begin{align*}
& B(\svec,\tvec,X) = B(\svec,\tvec)\, 
\exp\biggl( -\frac{F}{S} - \frac{3F^2}{2S^3}
+ O \biggl( \frac{\deltamax F}{S^2}
+ \frac{F^3}{S^5}\biggr) \biggr).
\end{align*}
\end{thm}

The first step in the proof of Theorem~\ref{thm:bip-avoid} is to show that under our assumptions, we do not expect
many edges of $X$ to appear in a typical element of $\mathcal{B}(\svec,\tvec)$.  Define
\[ N_0 = \lceil\max\{ \log S , 42F/S \}\rceil.\] 

\begin{lemma} \label{initial-bipartite}
Suppose that $\smax \tmax = o(S)$, $\smax + \tmax = o(S/\log S)$ and $(\smax+\tmax) F=o(S^2)$.
The probability that a uniformly randomly chosen element of $\mathcal{B}(\svec,\tvec)$ contains more than $N_0$ edges of $X$  is $O(1/S^2)$. 

\end{lemma}
\begin{proof}
Let $f = N_0 +1$.
For any set $A\subseteq X$ of $f$ distinct edges of $X$, let $B_1(A)$ be the set of all bipartite graphs $G\in \mathcal{B}(\svec,\tvec)$ with $A\subseteq G$ and let $B_0(A)$ be the set of all bipartite graphs $G\in \mathcal{B}(\svec,\tvec)$ with $A\cap G = \emptyset$. 
Suppose $A=\{e_1,\ldots,e_f\}$ where $e_i = u_{j_i}v_{k_i}$ for each $i\in [f]$.
Consider the following switching operation. From a bipartite graph $G\in B_1(A)$:
\begin{itemize}
    \item Choose $f$ edges $\hat{e}_1,\ldots, \hat{e}_f$ of $G$, where $\hat{e}_i = u_{p_i}v_{q_i}$ with $p_i\in [m]$ and $q_i\in [n]$, for $i\in [f]$, such that $\hat{e}_1,\ldots, \hat{e}_f$ are pairwise disjoint and disjoint from all elements of $A$, and such that
    \[ \{ u_{j_i}v_{q_i}, u_{p_i}v_{k_i} : i\in [f]\} \, \cap \, G = \emptyset.\]
    \item Form a new bipartite graph $G'$ from $G$ by deleting the edges $\{ \hat{e}_1,\ldots, \hat{e}_f\}\cup A$ and inserting the edges
    $\{ u_{j_i} v_{q_i},\, u_{p_i}v_{k_i} \, : \,  i \in [f]\}$.
\end{itemize}
The resulting graph $G'$ belongs to $B_0(A)$.
For each graph $G\in B_1(A)$ there are at least  $\bigl(S-2\smax \tmax - 2(\smax + \tmax)f\bigr)^f$ choices of forward switchings. To see this, we can choose the edges $\hat{e}_i$ in order: when choosing $\hat{e}_i$ we must exclude up to $2 (\smax +\tmax )f$ choices which intersect an element of $A$ or which intersect one of the already-chosen edges
$\hat{e}_1,\ldots, \hat{e}_{i-1}$, and we must exclude up to $2\smax\tmax$ choices which have $u_{j_i}v_{q_i}\in G$ or $u_{p_i}v_{k_i}\in G$.

Conversely, there are at most $\prod_{u_i v_j \in A} s_i t_j$ ways to produce a graph $G\in B_1(A)$ using a reverse switching from a given graph $G'\in B_0(A)$, since for each element $e_i$ of $A$ we must choose a pair of edges of $G'$, one incident with each endvertex of $e_i$. 
It follows that for all $A\subseteq X$ with $|A|=f$, the probability that a uniformly randomly chosen element of
$\mathcal{B}(\svec,\tvec)$ contains all elements of $A$ as edges can be bounded
above as
\[ \frac{|B_1(A)|}{B(\svec,\tvec)} \leq \frac{|B_1(A)|}{|B_0(A)|} \leq \frac{\prod_{u_i v_j \in A} s_i t_j}{\bigl(S-2\smax\tmax - 2(\smax +\tmax) f\bigr)^f}.\]
Note that the lemma assumptions imply that $\smax\tmax + (\smax+\tmax)N_0=o(S)$.

Let $\binom{X}{f}$ denote the set of all subsets of $X$ of size $f$.
The desired probability is at most the expected number of sets of $f$ edges of $X$ which are contained in $G$, 
which is at most
\begin{align*}
\sum_{A \in \binom{X}{f}} \frac{|B_1(A)|}{|B_0(A)|}
&\le \sum_{A \in \binom{X}{f}}  \frac{\prod_{u_i v_j \in A} s_i t_j}{\bigl(S-2 \smax \tmax -2(\smax+\tmax)f\bigr)^f} \\
&\le \frac{1}{f!}\, \biggl(\frac{F}{S(1-o(1))}\biggr)^f 
\le \biggl(\frac{e F}{fS (1-o(1)}\biggr)^f
\le \biggl(\frac{e}{41}\biggr)^f
= O(1/S^2).
\end{align*}
These inequalities follow from the definition of $N_0$ and our assumptions, together with the combinatorial identity
\[ \sum_{1\leq i_1 < i_2 < \cdots < i_k\leq r} a_{i_1}a_{i_2}\cdots a_{i_k} \leq \frac{1}{k!} \biggl(\,\sum_{i\in [r]} a_i\biggr)^{\!k}\]
applied with $r=|X|$, $k=f$ and where $a_1,\ldots, a_r$ is a sequence formed by the elements of the multiset $\{ s_it_j \, : u_iv_j\in X\}$ in some order, respecting multiplicities.
\end{proof}

For $f =0,1,\ldots, N_0$, let $\mathcal{B}_{f} = \mathcal{B}_{f}(\svec, \tvec, X)$ be the set of all bipartite graphs in $\calB(\svec,\tvec)$ which contain exactly $f$ edges from $X$. Note that $B(\svec,\tvec,X) = |\mathcal{B}_0(\svec,\tvec,X)|$.
We use a switching argument to approximate the ratio of the sizes of consecutive sets $\mathcal{B}_f$ and $\mathcal{B}_{f-1}$.

We will make use of the following switching operations.
A \textit{forward switching}, designed to reduce the number of edges of $X$ contained in the graph by exactly one, proceeds as follows.
From  $G\in \mathcal{B}_f$, 
\begin{itemize}
    \item 
    Choose an edge $u_iv_j\in G\cap X$ and two edges
$u_av_c, u_bv_d\in G\setminus X$ such that $u_iv_c,u_av_d,\allowbreak u_b v_j\notin G\cup X$.
    \item Let $G'$ be the graph obtained from $G$ by replacing these three edges by $u_i v_c$, $u_a v_d$, $u_b v_j$.
\end{itemize}
This switching operation is shown in Figure~\ref{fig:BipartiteSwitching}.  
By construction, the switching produces a (simple) graph $G'\in\mathcal{B}_{f-1}$.
Note also that the conditions on the chosen edges imply that the six vertices involved in the switching
are distinct. (In particular this follows by considering the 6-cycle $u_iv_ju_bv_du_av_cu_i$, which alternates between edges and non-edges of $G$.)

\begin{figure}[htb]
\unitlength=1cm
\centering
\begin{picture}(12,3.6)
  \put(0.5,0){\includegraphics[scale=0.9]{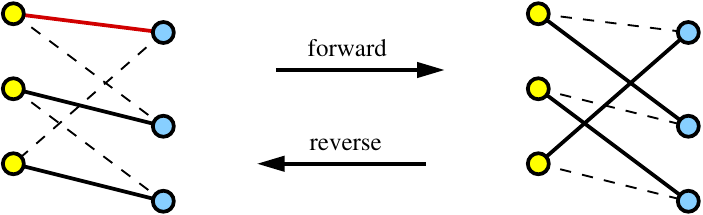}}
  \put(0.0,3.0){$u_i$}
  \put(0.0,1.85){$u_a$}
  \put(0.0,0.75){$u_b$}
  \put(3.3,2.7){$v_j$}
  \put(3.3,1.28){$v_c$}
  \put(3.3,0.14){$v_d$}
  \put(8.02,3.0){$u_i$}
  \put(8.02,1.85){$u_a$}
  \put(8.02,0.75){$u_b$}
  \put(11.3,2.7){$v_j$}
  \put(11.3,1.28){$v_c$}
  \put(11.3,0.14){$v_d$}
\end{picture}
   \caption{Switching to remove an edge of $X$}
   \label{fig:BipartiteSwitching}
\end{figure}
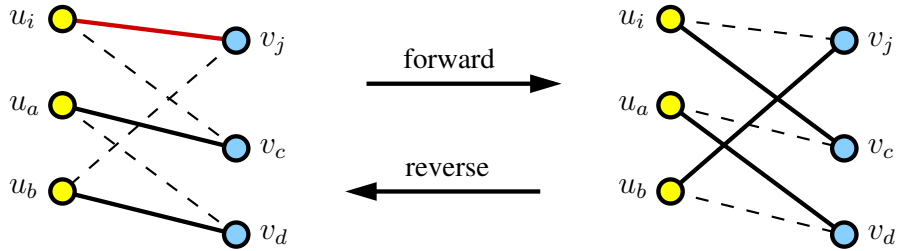

A \textit{reverse switching} is the reverse of the forward switching. It proceeds as follows: starting with a graph $G'\in\mathcal{B}_{f-1}$,
\begin{itemize}
\item Choose $u_iv_j\in X\setminus G'$, then choose an edge $u_iv_c\in G'\setminus X$ incident with $u_i$,
an edge $u_bv_j\in G'\setminus X$ incident with $v_j$, and a third edge $u_av_d\in G'\setminus X$, such that $u_av_c,\, u_bv_d\notin G'\cup X$.
\item Let $G$ be the graph obtained from $G'$ by deleting the edges $u_iv_c$, $u_bv_j$, $u_av_d$
and inserting the edges $u_iv_j$, $u_av_c$, $u_bv_d$.
\end{itemize}
By construction, the reverse switching produces a (simple) graph $G\in\mathcal{B}_f$.
Again, the conditions of the switching imply that the six vertices are distinct.

\begin{lemma}
\label{forbidden-ratio-bipartite}
Suppose that the assumptions of
Theorem~\textit{\ref{thm:bip-avoid}} hold.
Then
\[
|\mathcal{B}_{f}| = | \mathcal{B}_{f-1} | \, \frac{F(S-f+1) - O\bigl(\deltamax (F+(f-1)S)\bigr)}
  {
f \bigl( (S-f)^2 - F \bigr)
}
\]
uniformly for all $f \in [N_0]$ such that $\mathcal{B}_{f-1}$ is nonempty.
\end{lemma}

\begin{proof}
Given $G \in \mathcal{B}_{f}$, let $N = N(G)$ be the number of forward switchings which can be applied to $G$. There are $f$ choices for the edge $u_i v_j\in G\cap X$, and at most $(S-f)^2$ choices for the two edges $u_a v_c, u_b v_d\in G\setminus X$.
Of these choices, 
\[ fF+ O\bigl(f^2(\xmax\tmax + \smax\ymax) \bigr)\]
have $u_av_d\in X$, where the error term comes from
the possibility that one or both of the edges $u_av_c$, $u_bv_d$ also belong to $X$.
This leads to the upper bound
\begin{equation}
N\leq f\bigl( (S-f)^2 - F + O\bigl(f(\xmax\tmax + \smax\ymax) \bigr)\bigr).
\label{eq:N-upper}
\end{equation}
Next we consider the following (possibly overlapping) choices which violate a condition of the forward switching.  (Recall that these conditions imply that the six vertices are
distinct, so we do not need to consider this case separately.) 
\begin{itemize}
\item \textit{A double edge is created}:\ First suppose that $u_i v_c \in G$. There are $f$ choices for $u_i v_j$, then at most $\smax-1$ choices for $v_c$, then at most $\tmax-1$ choices for $u_a$, and at most $S$ choices for $u_b v_d$. The same estimate holds
if $u_av_d\in G$ or $u_bv_j\in G$.  Hence there are $O(f\smax\tmax S)$ choices in this case. 
\item \textit{An edge of $X$ is created}:\ We have already considered the case
that $u_av_d\in X$ above. Next, suppose that $u_i v_c \in X$. 
There are $f$ choices for $u_i v_j$, then at most $\xmax-1$ choices for $v_c$, then at most $\tmax-1$ choices for $u_a$, and at most $S$ choices for $u_b v_d$. So there are at most $f\tmax\xmax S$ such choices, and similarly at most $f\smax\ymax S$ choices with $u_bv_j\in X$. 
\end{itemize}
Comparing the number of exclusions to the
upper bound (\ref{eq:N-upper}), we obtain
\begin{equation} \label{forbidden-forward-count}
N = f \bigl( (S-f)^2 - F +  O\bigl(\deltamax S \bigr)\bigr) 
= f \bigl( (S-f)^2 - F\bigr)\bigl( 1 + O(\deltamax/S)\bigr).
\end{equation}

\medskip

Now we analyse the reverse switching. Given $G' \in \mathcal{B}_{f-1}$, let $N' = N'(G')$ be the number of reverse switchings which can be applied to $G'$. There are $F$ ways to choose $u_i v_c, u_b v_j\in G'$ such that $u_iv_j \in X$, and there are  at most $S-(f-1)$ choices for the edge $u_a v_d\in G'\setminus X$.
The following (possibly overlapping) choices violate a condition of the reverse switching.
\begin{itemize}
\item \textit{More than one edge of $X$ is created}:\ Suppose that $u_a v_c \in X$. There are $F$ ways to choose $u_i v_c$ and $u_b v_j$, then at most $\ymax$ choices for $u_a$, then at most $\smax-1$ choices for $v_d$. Hence there are at most $O(\smax\ymax F)$ such choices, and similarly there are at most $O(\xmax\tmax F)$
choices with $u_bv_d\in X$.
\item \textit{An edge of $X$ is removed}: Since we ensured that $u_av_d\notin X$, this case can only arise if $u_iv_c\in X$ or $u_bv_j\in X$.
There are at most $O\bigl((f-1)(\xmax\tmax + \smax\ymax)S\bigr)$ such choices.
\item \textit{A double edge is created}:\ First suppose that $u_i v_j \in G'$. There are $f-1$ choices for $u_i v_j$, then at most $\smax-1$ choices for $v_c$ and at most $\tmax-1$ choices for $u_b$, and at most $S$ choices for $u_a v_d$. This gives at most $O((f-1)\smax \tmax S)$ such choices.  Next, if $u_av_c\in G'$ then there are $F$ ways to choose $u_iv_c$ and $u_bv_j$, then at most $\tmax-1$ choices for $u_a$ from the neighbourhood of $v_c$, and at most $\smax-1$ choices for $v_d$ from the neighbourhood of $u_a$.
Hence there are at most $O(\smax\tmax F)$ such choices, and the same estimate holds if $u_bv_d\in G'$.
\end{itemize}

By subtracting the number of bad choices and comparing with the upper bound, we have
\begin{equation} \label{forbidden-reverse-count}
  N' = F(S-f+1) - O\bigl(\deltamax (F+(f-1)S)\bigr).
\end{equation}

The proof is completed by observing that $\sum_{G\in \mathcal{B}_f} N(G) = \sum_{G'\in\mathcal{B}_{f-1}} N'(G')$, applying (\ref{forbidden-forward-count}) and (\ref{forbidden-reverse-count}) and using our assumptions.
\end{proof}

\bigskip

To combine these estimates will require the following technical lemma. 

\begin{lemma}[\protect{\cite[Corollary 4.5]{GMW},\cite[Lemma~2.4]{BG2016}}]
\label{summation-lemma}
Let $N \ge 2$ be an integer and, for $1 \le i \le N$, let real numbers $A(i), C(i)$ be given such that $A(i) \ge 0$ and $A(i) - (i-1)C(i) \ge 0$. Define $A_1 = \min_{i\in [N]} A(i), A_2 = \max_{i\in [N]} A(i), C_1 = \min_{i\in [N]} C(i), C_2 = \max_{i\in [N]} C(i)$. Suppose that there exists a real number $\hat{c}$ with $0 < \hat{c} < \tfrac13$ such that $\max\{ A_2/N, |C_1|, |C_2| \} \le \hat{c}$. Define $n_0, \ldots n_N$ by $n_0 = 1$ and
\[
n_i = \frac{1}{i} \bigl( A(i) - (i-1)C(i) \bigr) n_{i-1}
\]
for $i\in [N]$. Then
\[
\Sigma_1 \le \sum_{i\in [N]} n_i \le \Sigma_2,
\]
where
\begin{align*}
\Sigma_1 &= \exp(A_1 - \tfrac12 A_1 C_2) - (2e\hat{c})^N, \\
\Sigma_2 &= \exp(A_2 - \tfrac12 A_2 C_1 + \tfrac12 A_2 C_1^2) + (2e\hat{c})^N.
\end{align*}
\end{lemma}

\medskip

\begin{lemma}
\label{forbidden-sum-bipartite}
Let $X\subseteq U\times V$ be a 
bipartite graph on $U\cup V$.
Suppose that the assumptions of Theorem~$\ref{thm:bip-avoid}$ hold and that $F\geq 1$.
Then
\[
\sum_{f=0}^{N_0} |\mathcal{B}_{f}| 
= |\mathcal{B}_{0}| \,\exp \biggl( \frac{F}{S} + \frac{3F^2}{2S^3}+ O \biggl( \frac{\deltamax F}{S^2}
+ \frac{F^3}{S^5}\biggr) \biggr).
\]
\end{lemma}

\begin{proof}
By \eqref{forbidden-forward-count}, any $G \in \mathcal{B}_f$ can be converted to some $G' \in \mathcal{B}_{f-1}$ using a forward switching.  Therefore, if $\mathcal{B}_0$ is empty then $\mathcal{B}_{f}$ is empty for all $f\in [N_0]$.
The lemma holds in this case since both sides of the expression equal 0. So we assume that $\mathcal{B}_0\neq \emptyset$.

Define
\[
   A_0 = \frac{FS}{(S-1)^2-F}, \quad 
   C_0 = - \frac{F(S^2+F-1)}{( (S-1)^2-F)^2}.
\]
Then
\begin{align*}
   \frac{F(S-f+1)}{(S-f)^2-F} 
   &= A_0 - (f-1)C_0 + O\biggl(\frac{(f-1)^2 F}{S^3}\biggr)\\
     &= A_0 - (f-1)C_0 + O\biggl(\frac{(f-1) N_0 F}{S^3}\biggr),
\end{align*}
as can be seen by taking the Taylor expansion of the left hand side at $f=1$.   It follows from Lemma~\ref{forbidden-ratio-bipartite} that
\[ |\mathcal{B}_f| = \frac{|\mathcal{B}_{f-1}|}{f}\biggl(A_0 - (f-1)C_0 + O\biggl(\frac{\deltamax F}{S^2}
  + (f-1)\biggl(\frac{\deltamax}{S} + \frac{N_0  F}{S^3}\biggr)\biggr)\biggr)\]
uniformly for any $f\in [N_0]$ such that $\mathcal{B}_{f-1}$ is nonempty.
Hence we can define a real number $\alpha_f$
for all $f\in [N_0]$, such that  
\begin{align}\label{eq:recurrence}
|\mathcal{B}_{f}| = \frac{|\mathcal{B}_{f-1}|}{f}\biggl(A_0 + \frac{\alpha_f\, \deltamax F}{S^2}
  - (f-1)\biggl( C_0 - \alpha_f\biggl(\frac{\deltamax}{S} + \frac{N_0  F}{S^3}\biggr)\biggr)\biggr), 
\end{align}
where
$|\alpha_f|$ is bounded independently of $f$ and $S$.  In particular, if $\mathcal{B}_{f-1}$ is nonempty then $\alpha_f$ is uniquely defined by (\ref{eq:recurrence}),
while if $\mathcal{B}_{f-1}$ is empty then we let $\alpha_f=0$.  
Next, for $1 \le f \le N_0$, define
\begin{align*}
A(f) = A_0 + \frac{\alpha_f\, \deltamax F}{S^2},\quad
C(f) = C_0 - \alpha_f\biggl(\frac{\deltamax}{S} + \frac{N_0  F}{S^3}\biggr).
\end{align*}
Then for all $1 \le f \le N_0$ we can rewrite (\ref{eq:recurrence}) as 
\begin{equation} \label{forbidden-ratio-bipartite-ac}
|\mathcal{B}_f| = \frac{1}{f} \bigl( A(f) - (f-1)C(f) \bigr) \,|\mathcal{B}_{f-1}|.
\end{equation}

We wish to apply Lemma~\ref{summation-lemma}, so we must check that the conditions of that lemma hold. First we claim that $A(f)-(f-1)C(f)\ge 0$ for all $f\in [N_0]$. 
If $\mathcal{B}_{f-1}$ is nonempty 
then (\ref{forbidden-ratio-bipartite-ac}) implies that $A(f)-(f-1)C(f)\ge 0$, since
$|\mathcal{B}_f|\geq 0$.  Otherwise $\mathcal{B}_{f-1}$ is empty, and hence
$A(f)=A_0$ and $C(f)=C_0$.
Since $A_0\ge 0$ and $C_0\le 0$, it follows that $A_0-(f-1)C_0\ge 0$
for all~$f\in [N_0]$, and the first claim is established.
Next, the assumption that $\deltamax = o(S)$ implies that $A(f)\geq 0$ for large enough $S$, since $A_0 = \Theta(F/S)$.

Define $A_1, A_2, C_1, C_2$ to be the minimum and maximum of $A(f)$ and $C(f)$ over $f\in [N_0]$, as in Lemma~\ref{summation-lemma}, and set $\hat{c} = 1/41$. Since $A_2 = \frac{F}{S}(1 + o(1))$ 
and $C_1,C_2 = o(1)$, under our assumptions, 
we have for $S$ sufficiently large that
\[
\max\{A_2/N_0, |C_1|, |C_2|\} \le A_2/N_0 
< \hat{c}.
\]
Therefore Lemma~\ref{summation-lemma} applies. 

Direct calculations show that
\[
A_0 = \frac{F}{S} + \frac{F^2}{S^3} + O\biggl(\frac{F}{S^2} + \frac{F^3}{S^5}\biggr),
\qquad
A_0C_0 = - \frac{F^2}{S^3} +  O\biggl(\frac{F}{S^2} + \frac{F^3}{S^5}\biggr).
\]
Hence
\begin{align} 
A_2 - \dfrac{1}{2} A_2 C_1 &= \biggl( A_0 + O\biggl(\frac{\deltamax F}{S^2}\biggr)\biggr) 
\biggl( 1 - \dfrac{1}{2}\, C_0 + O\biggl(\frac{\deltamax}{S} + \frac{N_0  F}{S^3}\biggr)\biggr) \notag \\
&= A_0 - \dfrac{1}{2} A_0 C_0 + O\biggl( \frac{\deltamax F}{S^2} + \frac{F^3}{S^5}\biggr)\label{NX-step} \\
&= \frac{F}{S} + \frac{3F^2}{2S^3} + O\biggl(\frac{\deltamax F}{S^2} + \frac{F^3}{S^5}\biggr).\notag
\end{align}
The same expression holds for $A_1 - \dfrac{1}{2} A_1 C_2$, up to the stated error term.
Note that to obtain (\ref{NX-step}) in the case that $N_0 = \lceil \log S\rceil$, the additive error term $O(A_0N_0 F/S^3)$
is covered by $O(\deltamax F/S^2)$, using the fact that $\deltamax\geq \smax\tmax \geq 1$ for $S$ sufficiently large.

Since $A_2 C_1^2 = O(F^3/S^5)$, 
by combining the lower and upper bounds from Lemma~\ref{summation-lemma} we conclude that
\[
\sum_{f=0}^{N_0} \frac{|\mathcal{B}_{f}|}{|\mathcal{B}_{0}|} = \exp \biggl( \frac{F}{S} + \frac{3F^2}{2S^3} + O\biggl( \frac{\deltamax F}{S^2} + \frac{F^3}{S^5}\biggr)\biggr) + O\bigl((2e/41)^{N_0}\bigr).
\]
Finally, $(2e/41)^{N_0} \le (1/e^2)^{\log S} \le 1/S^2$. Since the sum we are estimating is at least equal to one, this additive error term can be brought inside the exponential, completing the proof.
\end{proof}

We can now prove the main result of this section.

\begin{proof}[Proof of Theorem~$\ref{thm:bip-avoid}$]
If $F=0$ then $\mathcal{B}(\svec,\tvec,X)=\mathcal{B}(\svec,\tvec)$ and the theorem is true.
So we can assume that $F>0$.
Lemma~\ref{initial-bipartite} implies that
\[ B(\svec,\tvec) = \bigl(1 + O(S^{-2})\bigr)\,
  \sum_{f=0}^{N_0} |\mathcal{B}_f|\]
since the sets $\mathcal{B}_f$ are disjoint. Combining this with Lemma~\ref{forbidden-sum-bipartite} gives
\[ B(\svec,\tvec)=|\mathcal{B}_0|\,
\exp \biggl( \frac{F}{S} + \frac{3F^2}{2S^3}+ O\biggl( \frac{\deltamax F}{S^2} + \frac{F^3}{S^5}\biggr)\biggr),
\]
since the term $O(S^{-2})$ is covered by $O(\deltamax F/S^2)$.
Recalling that $B(\svec,\tvec,X) = |\mathcal{B}_0|$,  the result follows.
\end{proof}

\subsection{Bipartite graph applications}

Given a bipartite graph $G\in \mathcal{B}(\svec,\tvec)$ which contains all edges of $X$, we may delete all edges of $X$ to obtain a graph $G'\in \mathcal{B}(\svec-\xvec,\tvec-\yvec,X)$. This operation is a bijection,
and hence
\begin{equation}
|\{ G\in \mathcal{B}(\svec,\tvec) \, : \, X\subseteq G\}| = B(\svec-\xvec,\tvec-\yvec,X).
\label{bijection}
\end{equation}
 Dividing the above by $B(\svec,\tvec)$, we obtain the probability $P(\svec,\tvec,X)$ that a uniformly random element of $B(\svec,\tvec)$ contains $X$ as a subgraph.
 This allows the calculation of expected values, after summing over relevant choices of $X$.

 McKay~\cite[Theorem~3.5]{McKay1981} gave deterministic (non-asymptotic) upper and lower bounds for $P(\svec,\tvec,X)$. This work has been updated recently by Larkin, McKay and Tian~\cite[Section~5]{LMT}.  An asymptotic expression for $P(\svec,\tvec,X)$ for dense degrees was given in~\cite[Theorem~2.2]{GMbip}.
 Liebenau and Wormald~\cite[Theorem~1.4]{LW2022} gave a very precise
 formula for the probability that a randomly chosen element of $B(\svec,\tvec)$
 contains a specified edge, in the near-regular case.

\begin{cor}
Let $X\subseteq U\times V$ be a 
bipartite graph on $U\cup V$
and define the parameter
\[ \widehat{F} = \widehat{F}(X) = \sum_{u_iv_j\in X} (s_i - x_i)(t_j-y_j), \]
where $(\xvec,\yvec)$ is the degree sequence of $X$.
Define $\hat{s}_{\max} = \max_{i\in [m]} (s_i-x_i)$, $\hat{t}_{\max} = \max_{j\in [n]} (t_j-y_j)$ and $\widehat S=S-\abs{X}$,
and let 
\[ \hat{\delta}_{\max} =\hat{s}_{\max}\hat{t}_{\max}+\hat{s}_{\max}\ymax+\xmax\hat{t}_{\max}. \]
Suppose that $\hat{s}_{\max} + \hat{t}_{\max} = o\bigl(\widehat S/\log \widehat S\bigr)$,
$\hat{\delta}_{\max} = o(\widehat S)$, $\hat{\delta}_{\max}\widehat{F} = o\bigl(\widehat S^{2}\bigr)$,
$\widehat{F} = o\bigl( \widehat S^{\,5/3}\bigr)$.
Then, as $\widehat S\to\infty$, the probability that a uniformly random element of $\mathcal{B}(\svec,\tvec)$ contains every edge of $X$ is
\[ \frac{B(\svec-\xvec,\tvec-\yvec)}{B(\svec,\tvec)}\, \exp\biggl( - \frac{\widehat{F}}{\widehat S} 
- \frac{3\widehat{F}^{2}}{2\widehat S^{3}} 
+ O\biggl(\frac{\hat{\delta}_{\max} \widehat{F}}{\widehat S^{2}} 
+ \frac{\widehat{F}^{3}}{\widehat S^{5}}\biggr)\biggr).\]
\end{cor}

\begin{proof}
 The result follows using (\ref{bijection}), applying Theorem~\ref{thm:bip-avoid} to approximate the cardinality of $\mathcal{B}(\svec-\xvec,\tvec-\yvec,X)$.
\end{proof}

More general applications are also possible.
Let $X$ and $Z$ be disjoint subgraphs of the complete bipartite graph on $U\cup V$.
As usual we let $(\xvec,\yvec)$ denote the degree sequence of $X$ and define $(\wvec,\zvec)$ to be the degree sequence of $Z$.
Then the probability that a random chosen element of $\mathcal{B}(\svec,\tvec)$ contains every edge of $X$ and no edge of $Z$ is given by
\[ \frac{B(\svec-\xvec,\tvec-\yvec,X\cup Z)}{B(\svec,\tvec)}, \]
which can be approximated using Theorem~\ref{thm:bip-avoid} under the
required sparsity conditions on $\svec,\tvec,\xvec,\yvec,\wvec,\zvec$.

\section{Digraphs without loops}\label{s:digraphs}

We will make frequent uses of the undirected bipartite graph
representation of a digraph.
An example is shown in Figure~\ref{BipartiteFig}.
\begin{figure}[htb]
  \[ \includegraphics[scale=0.9]{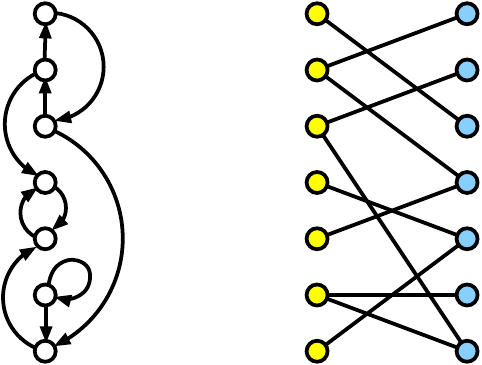} \]
   \caption{A digraph and its associated bipartite graph.}
   \label{BipartiteFig}
\end{figure}
Each vertex $w_i$ of a digraph provides two vertices
$u_i,v_i$ to its associated bipartite graph, while
each edge $w_iw_j$ of the digraph provides the
edge $u_iv_j$ to the bipartite graph.
Thus, a loop $w_iw_i$ in the digraph corresponds to an edge
$u_iv_i$ in the bipartite graph, while a 2-cycle
$w_i w_j w_i$ in the digraph corresponds to a pair of
edges $u_iv_j,u_jv_i$ in the bipartite graph.
Due to this correspondence, we will freely use the words
``loop'' and ``2-cycle'' when referring to the bipartite graph.

Using this bipartite representation of digraphs,
Greenhill and McKay~\cite[Theorem~3.1]{GMbip} gave a formula for the number of loop-free digraphs with specified degrees which avoid some set $X$ of specified edges, where the degrees are very dense and $|X|$ may be slightly superlinear in $n$. 
Liebenau and Wormald~\cite[Theorem~1.1]{LW2022} provided a formula for the number of loop-free digraphs with specified degrees, which holds for near-regular degree sequences of a wide range of densities. To the best of our knowledge, a formula for the number of loop-free digraphs 
with specified sparse degrees has not been written down, though it follows
as an easy corollary from McKay~\cite[Theorem 4.6]{silver}, for example.
In this section, we will apply Theorem~\ref{thm:bip-avoid} to present an asymptotic formula
for the number of sufficiently sparse loop-free digraphs with specified degrees.

To avoid trivial cases we will assume there are no isolated
vertices, which means that $s_i+t_i\ge 1$ for all~$i$.
This implies $S\ge \frac12 n$.
This assumption does not effect the validity of our
enumerations, since both the exact values and our
approximations will be independent of the addition
of isolated vertices.
Define
\begin{equation}\label{eq:W-def}
    W = \sum_{i\in [n]} s_it_i.
\end{equation}

\medskip

\begin{thm}\label{loopprob}
As $S\to\infty$, suppose that $\smax+\tmax=o(S/\log S)$
and $\smax\tmax W=o(S^2)$.
Then the probability that a random digraph with degrees $\svec,\tvec$ has no loops is
\[
   \exp\biggl( -\frac WS + O\biggl(\frac{\smax\tmax
    W}{S^2}\biggr)\biggr).
\]
\end{thm}
\begin{proof}
  No loops are possible if $W=0$, so assume $W\ge 1$.
  Now we can apply Theorem~\ref{thm:bip-avoid}
  using $m=n$ and $F=W$.
  Since $\xmax=\ymax=1$, we have $\deltamax=\Theta(\smax\tmax)$.

  By the definition of $W$, we have $W\le \smax S$
  and $W\le\tmax S$, which together imply
  $W\le (\smax\tmax)^{1/2} S$.
  This gives $W^3/S^5= o(\smax\tmax W/S^2)=o(1)$,
  which satisfies the assumption $F=o(S^{5/3})$ of
  Theorem~\ref{thm:bip-avoid} and shows that we
  can discard the error term $O(F^3/S^5)$.
  Finally, $W^2/S^3 \le \smax\tmax W/S^2$ so the
  term $F^2/(2S^3)$ of Theorem~\ref{thm:bip-avoid}
  also lies within the error term.
\end{proof}
\medskip

\begin{cor}\label{loopfreethm}
   Suppose $S\to\infty$ and $\smax\tmax=o(S^{2/3})$.
   Let $R(\svec,\tvec)$ be
    the number of loop-free simple digraphs
   with degrees $\svec,\tvec$. Then
   \[
R(\svec,\tvec) = \frac{S!}{\prod_{i\in [n]} (s_i!\,t_i!)}
\exp\biggl( Q(\svec,\tvec)
- \frac{W}{S}
+ O\biggl(\frac{\smax^3\tmax^3}{S^2} +
    \frac{\smax\tmax W}{S^2}\biggr)\biggr).
\]
In particular, if $1\le d=o(n^{1/2})$ then the number of
loop-free regular digraphs of in-degree and out-degree $d$ is
\[
    \frac{(dn)!}{(d!)^{2n}}
    \exp\biggl( -\frac{d^2+1}{2} - \frac{d^3}{6n}
        + O\biggl(\frac{d^2}{n}\biggr)\biggr).
\]
\end{cor}
\begin{proof}
  This follows from Theorem~\ref{GMW} and
  Theorem~\ref{loopprob}, after noting that
  $\smax\tmax=o(S^{2/3})$ implies
  $\smax\tmax W/S^2=o(1)$ and $\smax+\tmax=o(S/\log S)$.
\end{proof}

The regular case of Corollary~\ref{loopfreethm} follows
from Hasheminezhad and McKay~\cite[Lemma~3.6]{HM23} with the weaker error term $O(d/n^{1/2})$.
Liebenau and Wormald~\cite[Theorem~1.1]{LW2022} obtained an
estimate for $R(\svec,\tvec)$ when the degrees
were not too far from equal and not very small
or very large.

\begin{cor}
Let $X$ be a loop-free digraph on the vertex set $\{ w_1,\ldots, w_n\}$ with degree sequence $(\xvec,\yvec)$.  Define
\[ F = \sum_{(w_i,w_j)\in X} s_it_j,\qquad
  \deltamax = \smax\tmax + \smax \ymax + \xmax\tmax\]
  and define $W$ as in $(\ref{eq:W-def})$.
Suppose that $\smax + \tmax = o(S/\log S)$, $\smax\tmax = o(S^{2/3})$, $\deltamax = o(S)$, $\deltamax (F + W)  = o(S^2)$, $F= o(S^{3/5})$.
Then the number of loop-free digraphs with degrees $(\svec$, $\tvec)$ which do not contain any edge from $X$ is given by
\[ R(\svec,\tvec)\, \exp\biggl( - \frac{F}{S} - \frac{3F^2}{2S^3} + O\biggl(\frac{\smax^3\tmax^3}{S^2} + \frac{\deltamax(F + W)}{S^2} + \frac{F^2(F + W)}{S^5} \biggr)\biggr).\]
\end{cor}

\begin{proof}
Note that Theorem~\ref{loopprob} gives the ratio
$R(\svec,\tvec)/B(\svec,\tvec)$.
Using this, the result follows by applying Theorem~\ref{thm:bip-avoid} with parameters $\svec$, $\tvec$ and $X\cup \{ u_iv_i \, : \, i\in [n]\}$, arguing as in the proof of Theorem~\ref{loopprob}.
\end{proof}

\section{Permanents of random 0-1 matrices}\label{s:permanents}

In this section we apply our results to determine the expected
permanent of a random matrix with given row and column sums,
which is equivalent to the expected number of perfect matchings
of a random balanced bipartite graph.
In the regular case (where all the row and column sums are
equal), we combine our calculations with previous work to
cover all densities.

For the sparse range which is our primary focus, milestones
include  O'Neil~\cite{ONeil} in the regular case
for row sums up to $(\log n)^{1/4-\eps}$ and
Bollob\'as and McKay~\cite{Boll} for row sums up to
$(\log n)^{1/3}$ including the irregular case.
The irregular case for row sums $o(n^{1/3})$
follows from McKay~\cite{silver} but is not stated explicitly there.

The quantity $R(\svec,\tvec)$ defined in Corollary~\ref{loopfreethm} can be interpreted as the number of
square 0-1 matrices with row sums $\svec$, column sums $\tvec$,
and zero trace.
Note that $S,S_2,S_3,T_2,T_3,W$ are all functions of $\svec,\tvec$.
If instead we want all diagonal entries to be equal to~1,
the count is $R(\svec{-}\jvec,\tvec{-}\jvec)$,
where $\jvec=(1,\ldots,1)$, provided
$\svec$ and $\tvec$ have no zero entries.

We will need the following lemma.
\begin{lemma}[{\cite[Lemma~3.1]{GIM}}]\label{GIMlemma}
  Suppose $u,v:\{1,\ldots,n\}\to\Reals$.
  Define the function $\Psi(\sigma)=\sum_{j\in [n]} u(j)v(\sigma_j)$
  for permutations $\sigma\in S_n$.
  Let $\X$ be a uniformly random permutation in~$S_n$.
  Define
  \begin{gather*}
     \bar u = \Dfrac1n\sum_{j\in [n]} u(j), \qquad
     \bar v = \Dfrac1n\sum_{j\in [n]} v(j); \\
     \alpha = \bigl(\max_j u(j)-\min_j u(j)\bigr)
              \bigl(\max_j v(j)-\min_j v(j)\bigr).
  \end{gather*}
\end{lemma}
Then
\begin{align*}
  \E \Psi(\X) &= n\bar u\bar v, \\
  \Var \Psi(\X) &= \Dfrac{1}{n-1}
    \sum_{j\in [n]} (u(j)-\bar u)^2\sum_{k\in [n]} (v(k)-\bar v)^2, \\
  \E e^{\Psi(\X)} &= \exp\bigl(\E\Psi(\X)+\dfrac12\Var\Psi(\X)+K\bigr) \\
  &\kern5em\text{for some $K$ with 
  $\abs K\le \dfrac32n\alpha^3 + 11n\alpha^4$.}    
\end{align*}

\begin{thm}\label{permanent}
   Assume that $\svec$ and $\tvec$ have no zero entries,
   that $S\ge (1+\delta) n$ for some fixed $\delta>0$,
   and that $\smax\tmax=o(S^{2/3})$.
   Then the expected permanent of an $n\times n$
   random 0-1 matrix with row
   sums $\svec$ and column sums $\tvec$ is
   \[
       \frac{\prod_{j\in [n]} (s_jt_j)}{\binom{S}{n}}
       \exp\biggl(-\frac{S-n}{n} + Q(\svec{-}\jvec,\tvec{-}\jvec) - Q(\svec,\tvec)
         + O\biggl(\frac{\smax^{3/2}\tmax^{3/2}}{S}\biggr)
         \biggr).
   \]
\end{thm}
\begin{proof}
  As noted above, the probability that a random matrix has only
  1s on the diagonal is $R(\svec{-}\jvec,\tvec{-}\jvec)/B(\svec,\tvec)$.
  Noting that $B(\svec,\tvec^\sigma)$ is independent of $\sigma$,
  where $\tvec^\sigma=(t_{\sigma_1},\ldots,t_{\sigma_n})$,
  the probability that there are only 1s on some other transversal
  $\{(i,\sigma_i)\}$ is
  $R(\svec{-}\jvec,\tvec^\sigma{-}\jvec)/B(\svec,\tvec)$.
  Since $Q(\svec,\tvec^\sigma)$ is independent of
  $\sigma$, we only need to contend with the term
  $-W/S$ in Corollary~\ref{loopfreethm}, plus the error term
  that contains~$W$.
  We will remove the latter problem by applying the uniform bound $W\le\smax^{1/2}\tmax^{1/2}S$.
  Applying Theorem~\ref{GMW} and Corollary~\ref{loopfreethm},
  we now have that the probability that a random matrix has
  only 1s on the transversal $\{(i,\sigma_i)\}$ is
  \begin{align*}
        &\frac{(S-n)!\prod_{j\in [n]} (s_jt_j)}{S!}
       \exp\biggl(Q(\svec{-}\jvec,\tvec{-}\jvec) - Q(\svec,\tvec)
         + \phi(\sigma)         
         + O\biggl(\frac{\smax^{3/2}\tmax^{3/2}}{S}\biggr)
         \biggr), \\
       & \kern3em\text{where~~}
         \phi(\sigma) = -\frac{W(\svec{-}\jvec,\tvec^\sigma{-}\jvec)}{S-n}.
  \end{align*}

  We now apply Lemma~\ref{GIMlemma} to estimate $\E e^{\phi(\X)}$
  when $\X$ is a random permutation.
  Let $u(j)=-\beta (s_j-1)$ and $v(j)=\beta (t_j-1)$,
  where $\beta=1/\sqrt{S-n}$.
  We have $\bar u=-\beta(S/n-1)$ and $\bar v=\beta (S/n-1)$,
  so $\E \phi(\X)=-\beta^2 n (S/n-1)^2 = -(S-n)/n$.
  We also have $\Var \phi(\X)=O(\smax\tmax/n)=O(\smax^{3/2}\tmax^{3/2}/S)$.
  In addition, $\alpha\le\beta^2\smax\tmax=\smax\tmax/(S-n)$,
  so the error term $K$ in Lemma~\ref{GIMlemma} also fits into our
  error term.
  Consequently, 
  $\E e^{\phi(\X)}$ is equal to $e^{-(S-n)/n}$ within our existing error term.
  Summing over all
  $n!$ permutations $\sigma$ completes the proof.
\end{proof}

We can also estimate the permanents of extremely dense 0-1 matrices.

\begin{thm}\label{complthm}
  Assume $\smax\tmax=o(S^{2/3})$ and $S=\Omega(n)$.
  Then the expected permanent of an $n\times n$ random 0-1 matrix with
  row sums $(n-s_1,\ldots,n-s_n)$ and column sums
  $(n-t_1,\ldots,n-t_n)$ is
  \[
      n!\,\exp\biggl( - \frac{S}{n} 
         + O\biggl(\frac{\smax^{3/2}\tmax^{3/2}}{S}\biggr) \biggr).
  \]
\end{thm}
\begin{proof}
   A transversal of ones in a matrix is a transversal of zeros
   in its 0-1 complement.
   Therefore, using the bound $W\le\smax^{1/2}\tmax^{1/2}S$
   in the error term of Corollary~\ref{loopfreethm},
   the expected permanent is
   \[
      n!\, \exp\biggl(O\biggl(\frac{\smax^{3/2}\tmax^{3/2}}{S}\biggr)\biggr)
        \E e^{\varphi(\X)},
      \text{~where ~}
       \varphi(\sigma)=-\Dfrac{W(\svec,\tvec^\sigma)}{S}
  \]
   and $\X$ is a random permutation.
   Now we can apply Lemma~\ref{GIMlemma} with
   $u(j)=-s_j$ and $v(j)=t_j/S$.
   We find $\E\Psi=-S/n$,
   $\Var\Psi=O(\smax\tmax/n)=O(\smax^{3/2}\tmax^{3/2}/S)$
   and $\alpha\le \smax\tmax/S$.
   Recalling that $S=\Omega(n)$, the proof is complete.
\end{proof}

In the case where most of the entries of the matrix
are equal to 1, a more direct analysis gives the
following.

\begin{thm}\label{complthm2}
  Assume that $S= \sum_{i\in [n]} s_i 
  =\sum_{i\in [n]} t_i = O(n)$.
  Then the permanent of every 0-1 matrix with
  row sums $(n-s_1,\ldots,n-s_n)$ and column sums
  $(n-t_1,\ldots,n-t_n)$ is
  \[
      n!\,\biggl( e^{-S/n} 
         + O\biggl(\frac{(\smax+\tmax)S}{n^2}\biggr)\biggr),
  \]
  where the error term can be taken inside the exponential
  if it is $o(1)$.
\end{thm}
\begin{proof}
The theorem is trivial for $S=0$ so assume $S>0$.
The permanent is the number of transversals that
meet no zero, which we estimate by
inclusion-exclusion on the events of meeting a zero.

For $k\ge 0$, let $m_k$ be the number of ways to choose an ordered sequence of $k$ zeros with no two in the same row or column.
The number of transversals that include a
particular set of $k$ zeros is $(n-k)!$ so,
by inclusion-exclusion, the permanent is
\[
  \sum_{k=0}^n \Dfrac{(-1)^k}{k!} (n-k)!\, m_k
  = n!\,\sum_{k=0}^n (-1)^k \frac{m_k}{k!\,(n)_k}.
\]
Each choice of a zero reduces the remaining choices
by at least one and at most $\smax+\tmax$, and so by
induction on $k$ we have that
\[
    \biggl( S - \binom k2(\smax+\tmax)\biggr) S^{k-1}
    \le m_k \le (S)_k
\]
for $0\le k\le n$ and, trivially, also for $k>n$.
Since $S<n$ by assumption,
$(S)_k/(n)_k\le S^k/n^k$ and so
\[
   \frac{m_k}{k!\,n^k}\le \frac{m_k}{k!\,(n)_k}
   \le \frac{S^k}{k!\,n^k}.
\]
Note that this inequality remains true for $k>n$
if we interpret the middle term as 0 in that case.
Consequently, the permanent is
$n!\,\bigl( e^{-S/n} + \varDelta(\smax,\tmax)\bigr)$,
where
\begin{align*}
   \Abs{\varDelta(\smax,\tmax)} \le 
   \sum_{k\ge 2}\frac{S^k-m_k}{k!\,n^k}
     &\le\sum_{k\ge 2}
     \frac{\binom k2(\smax+\tmax)S^{k-1}}{k!\,n^k} \\
     &= \frac{e^{S/n}\,(\smax+\tmax)S}{2n^2},
\end{align*}
which completes the proof since $e^{S/n}=O(1)$
by assumption.
\end{proof}

Denote by $M(n,d)$ the expected permanent of
an $n\times n$ random 0-1 matrix with each row and column sum equal to~$d$.
The combination of our theorems with previous results enables us to estimate $M(n,d)$
for all~$d$.

\begin{thm}\label{regulardet}
 Let $d=d(n)$ satisfy $2\le d\le n$.
   Then the expected permanent
   of an $n\times n$ random 0-1 matrix with all row and column sums
   equal to $d$ is
   \begin{equation}
       \label{eq:regperm}
       \frac{d^{2n}}{\binom{dn}{n}}
       \exp\Bigl( -\dfrac 12 + O(n^{-1/7})\Bigr)
       = \sqrt{\frac{2\pi (d-1)n}{d}}
        \biggl(\frac{(d-1)^{d-1}}{d^{d-2}}\biggr)^{\!n}
        \exp\Bigl( -\dfrac 12 + O(n^{-1/7})\Bigr).
    \end{equation}
\end{thm}
\begin{proof}
First, note that both expressions in~\eqref{eq:regperm}
are equal within their error terms.
For a real number $z$, let $\vec z$ denote
the vector $(z,z,\ldots,z)$ with $n$ components.
As
discussed in the proof of Theorem~\ref{complthm}, we have
\begin{equation}\label{eq:options}
 M(n,d) = \frac{R(\vec d-\vec 1,\vec d-\vec 1\,) \,n!}
   {B(\vec d,\vec d\,)}
   = \frac{R(\vec n-\vec d,\vec n-\vec d\,)\,n!}
   {B(\vec d,\vec d\,)},
\end{equation}
where $R$ is defined in Corollary~\ref{loopfreethm}.

The proof of the theorem proceeds in six ranges.
For $2\le d\le n^{1/3}$, the theorem
is a special case of Theorem~\ref{permanent}.
For $n^{1/3}\le d \le 2n/\log n$,
\cite[Theorem\ 1.1]{LW2022} applied to the
central expresssion in~\eqref{eq:options} with $\kappa=\frac1{70}$ gives
\begin{equation}\label{eq:LWvalue}
   M(n,d)=\frac{n!\, d^{2n} \binom{n^2}{nd}}
         { n^{2n} \binom{n(n-1}{n(d-1)}}
         \bigl(1 + O(n^{-1/7})\bigr),
\end{equation}
which matches~\eqref{eq:regperm}.

For $2n/\log n\le d\le n-2n/\log n$, we can
directly apply~\cite[Theorem\ 2.5]{GMbip} with $a=\frac13, b=\frac17$, which gives
\[
   M(n,d) = n!\,\lambda^n
     \exp\biggl(\frac{1-\lambda}{2\lambda}
      + O(n^{-1/7})\biggr),
\]
where $\lambda=d/n$, which matches~\eqref{eq:regperm}.
For $n-2n/\log n\le d\le d-n^{1/3}$, we apply 
\cite[Theorem\ 1.1]{LW2022} again, this time to the final expression in~\eqref{eq:options}.
It gives expression~\eqref{eq:LWvalue} again,
which still matches~\eqref{eq:regperm} despite
the different range of $d$.

Finally, for $n-n^{1/3}\le d\le n-1$ the
theorem is a special case of Theorem~\ref{complthm},
and for $d=n$ the exact value $M(n,n)=n!$
also matches~\eqref{eq:regperm}.
\end{proof}
    
\section{Oriented graphs}\label{s:oriented}

In this section we find an asymptotic formula for the number of
digraphs with degrees $\svec,\tvec$ that have no loops or 2-cycles,
under the stronger assumption that $\smax\tmax=o(S^{1/2})$.
These are commonly known as simple oriented graphs, since
they correspond to simple undirected graphs to which an
orientation has been assigned to each edge.

Recall that in the bipartite graph model, 2-cycles correspond to distinct
indices $i,j$ such that the edges $u_iv_j$ and $u_jv_i$ are both present.
We will use the notation $\tc ij$ to represent the 2-cycle $\{u_iv_j,u_jv_i\}$.

\begin{lemma}\label{2cycbound}
 Suppose $S\to\infty$ and $(\smax+\tmax)^2=o(S)$.
 Define the cutoff
 $N_1=\lceil\max\{\log S,\allowbreak 24W^2/S^2\}\rceil$.
 Then, with probability $1-O(S^{-2})$, a random loop-free
 bipartite graph has fewer than $N_1$ 2\mbox{-}cycles.
\end{lemma}
\begin{proof}
  Let $q=N_1$.  The probability that there are at least
  $q$ 2-cycles is at most the expected number of sets of $q$ 2-cycles.

  Let $D= \{\tc{i_1}{j_1}, \ldots, \tc{i_q}{j_q}\}$
  be a potential set of 2-cycles.
  Define $K$ to be the set of $2q$ edges of those 2-cycles.
  For $0 \le k \le q$,
  let $\mathcal{H}_k(D)$ be the set of loop-free bipartite graphs for which 2-cycles $\{\tc{i_1}{j_1}, \ldots, \tc{i_k}{j_k}\}$ are present
  and 2-cycles $\{\tc{i_{k+1}}{j_{k+1}}, \ldots, \tc{i_q}{j_q}\}$ are absent.

For a graph in $G\in\mathcal{H}_k(D)$ with $k \ge 1$, choose two distinct
edges $u_a v_b, u_cv_d\notin K$ such that
$u_{i_k}v_b$, $u_bv_{i_k}$,  $u_av_{i_k}$, $u_{i_k}v_a$,
$u_{j_k}v_d$, $u_dv_{j_k}$, $u_cv_{j_k}$ and $u_{j_k}v_c$ are
not edges of~$G$.

Then remove $u_{i_k} v_{j_k}$, $u_{j_k} v_{i_k}$,  $u_a v_b$ and $u_c v_d$ and insert $u_{i_k} v_b$, $u_{j_k}v_d$,  $u_a v_{i_k}$ and $u_c v_{j_k}$.
Since the 2-cycle $\tc{i_k}{j_k}$ is lost and no other 2-cycles
in $\{\tc{i_1}{j_1}, \ldots,\tc{i_q}{j_q}\}$ are
either destroyed or created,
this gives a graph in $\mathcal{H}_{k-1}$.
There are at least $S-2q$ choices of $u_av_b$ that are not
in~$K$.
In at most $2\smax\tmax$ cases
$u_{i_k}v_b$ or $u_av_{i_k}$ are edges, and in at most
$\smax^2+\tmax^2$ cases $u_{i_k}v_a$ or $u_bv_{i_k}$ are edges.
Thus the number of choices is at least $S-2q-(\smax+\tmax)^2$.
The number of choices of $u_cv_d$ has the same bound except
that we must not choose $u_av_b$ again.
Thus, the total number of choices for the switching is at
least $(S-2q-1-(\smax+\tmax)^2)^2$.

For the inverse operation, we only need an upper bound. 
We can choose edges $u_{i_k} v_b$, $u_{j_k} v_d$,  $u_a v_{i_k}$ and $u_c v_{j_k}$ in at most $s_{i_k}t_{i_k}s_{j_k}t_{j_k}$ ways.

Multiplying these ratios together, we find that:
\[
\mathbb{P}(D \subseteq G) \le \frac{|\mathcal{H}_d(D)|}{|\mathcal{H}_0(D)|}
\le \frac{\prod_{ij \in D}s_it_is_j t_j }{(S-2q-1-(\smax+\tmax)^2)^{2q}}.
\]

To sum over $D$, note that $\sum_{\card D=q}\prod_{ij \in D}s_it_i s_j t_j$ is the coefficient of $x^q$ in
$\prod_{i<j} (1 +s_it_i s_j t_j x),$
which is bounded by the coefficient of $x^q$ in $\prod_{i<j} e^{s_it_is_j t_j x} < e^{W^2x/2}$, which is $\frac{W^{2q}}{2^q q!}$.
Also note that $2q+1 +(\smax+\tmax)^2=o(S)$ and $q!>(q/e)^q$.
Since $q\ge 24 W^2/S^2$ and $q\ge\log S$,
the expected number of sets of $q$ 2-cycles is at most
\[
   \biggl(\frac{3W^2}{2qS^2}\biggr)^{\!q}
   \le 8^{-q} \le 8^{-\log S} < S^{-2},
\]
as desired.
\end{proof}

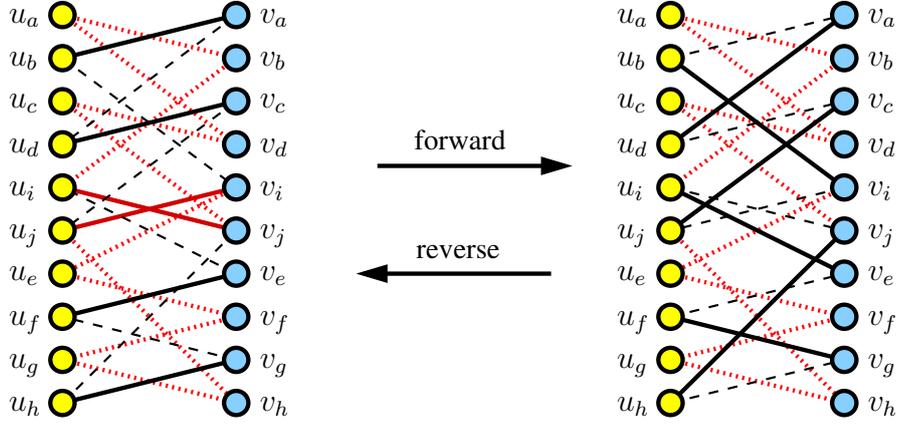
\begin{figure}[htb]
\unitlength=1cm
\centering
\begin{picture}(12,6)
  \put(0.5,0){\includegraphics[scale=0.9]{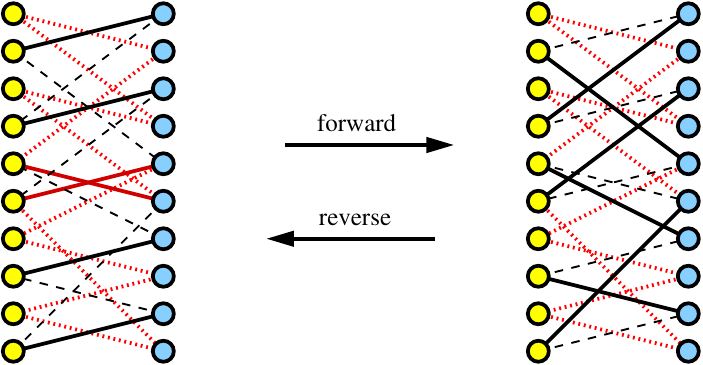}}
  \put(0.0,5.25){$u_a$}  \put(8.0,5.25){$u_a$}
  \put(0.0,4.68){$u_b$}  \put(8.0,4.68){$u_b$}
  \put(0.0,4.12){$u_c$}  \put(8.0,4.12){$u_c$}
  \put(0.0,3.53){$u_d$}  \put(8.0,3.53){$u_d$}
  \put(0.0,2.96){$u_i$}  \put(8.0,2.96){$u_i$}
  \put(0.0,2.39){$u_j$}  \put(8.0,2.39){$u_j$}
  \put(0.0,1.83){$u_e$}  \put(8.0,1.83){$u_e$}
  \put(0.0,1.24){$u_f$}  \put(8.0,1.24){$u_f$}
  \put(0.0,0.67){$u_g$}  \put(8.0,0.67){$u_g$}
  \put(0.0,0.10){$u_h$}  \put(8.0,0.10){$u_h$}
  \put(3.3,5.25){$v_a$}  \put(11.3,5.25){$v_a$}
  \put(3.3,4.68){$v_b$}  \put(11.3,4.68){$v_b$}
  \put(3.3,4.12){$v_c$}  \put(11.3,4.12){$v_c$}
  \put(3.3,3.53){$v_d$}  \put(11.3,3.53){$v_d$}
  \put(3.3,2.96){$v_i$}  \put(11.3,2.96){$v_i$}
  \put(3.3,2.39){$v_j$}  \put(11.3,2.39){$v_j$}
  \put(3.3,1.83){$v_e$}  \put(11.3,1.83){$v_e$}
  \put(3.3,1.24){$v_f$}  \put(11.3,1.24){$v_f$}
  \put(3.3,0.67){$v_g$}  \put(11.3,0.67){$v_g$}
  \put(3.3,0.10){$v_h$}  \put(11.3,0.10){$v_h$}
\end{picture}
   \caption{Switching to remove a 2-cycle.}
   \label{TwoCycSwitching}
\end{figure}

A switching operation that removes one 2-cycle is shown in
Figure~\ref{TwoCycSwitching}.
Define two sets of vertex pairs
$E_1=\{u_bv_a,u_dv_c,u_iv_j,u_jv_i,u_fv_e,u_hv_g\}$ and
$E_2=\{u_jv_c,u_bv_i,u_dv_a,u_iv_e,u_fv_g,\allowbreak u_hv_j\}$.
The 10 indices $a,b,\ldots,i,j$ must be distinct.
For the forward switching, the pairs in $E_1$ are edges and those in $E_2$ are not edges.
The switching consists of removing $E_1$ and
inserting $E_2$.
The reverse switching is the inverse operation.

The choice of the 10 unique indices is restricted so that
no 2-cycles are either destroyed or created,
except that the forward switching destroys the 2-cycle $\tc ij$.
This requires that none of the vertex pairs
$E_3=\{u_av_b,u_cv_d,u_ev_f,u_gv_h,
u_av_d,u_cv_j,\allowbreak u_iv_b,u_jv_h,u_ev_i,u_gv_f\}$,
shown as red dotted lines in the figure, may be edges.

Let $\calT_q$ denote the set of loop-free bipartite graphs with
degrees $\svec,\tvec$ which contain exactly $q$ 2-cycles.

\begin{thm}\label{2cyclemma}
  Suppose $\smax^2+\tmax^2=o(S)$
  and $\smax\tmax(\smax+\tmax)W=o(S^2)$.
  Then the probability that a uniformly random
  loop-free digraph with degree sequence $(\svec,\tvec)$
  has no 2-cycles is
  \[
     \exp\biggl( -\frac{W^2}{2S^2} + O\biggl(
        \frac{\smax\tmax(\smax+\tmax)W}{S^2}\biggr)\biggr).
  \]
\end{thm}

\begin{proof}
The theorem is trivially true when $W=0$, so assume
$W>0$.

Suppose $1\le q\le N_1$ and let $G\in\calT_q$.
We can choose the 2-cycle $\tc ij$ in $q$ ways.
As we choose each of the other four edges, there are $S-2q+O(\smax+\tmax)$
choices that do not lie in a 2-cycle and are not adjacent to a previously-chosen edge.
We can then bound the other forbidden cases as we choose each edge in the following order:
\begin{itemize}\itemsep=0pt
\item[$u_bv_a$:] $O(\smax\tmax)$ for $u_bv_i\in G$,
  $O(\smax^2)$ for $u_iv_b\in G$;
\item[$u_dv_c$:] $O(\smax\tmax)$ for $u_jv_c\in G$ or $u_dv_a\in G$,
  $O(\smax^2)$ for $u_av_d\in G$,\\ $O(\tmax^2)$ for $u_cv_j\in G$;
\item[$u_fv_e$:] $O(\smax\tmax)$ for $u_iv_e\in G$, $O(\tmax^2)$ for $u_ev_i\in G$;
\item[$u_hv_g$:] $O(\smax\tmax)$ for $u_hv_j\in G$ or $u_fv_g\in G$,
  $O(s^2)$ for $u_jv_h\in G$,\\ $O(\tmax^2)$ for $u_gv_f\in G$.
\end{itemize}
Consequently, the number of forward switchings is
\[
   N_F = q \bigl(S - 2q +O(\smax^2+\tmax^2)\bigr)^4.
\]

Next, suppose $1\le q\le N_1$ and let
$G'\in\calT_{q-1}$.
There are $W$ ways to choose $i$ and edges
$u_iv_e$ and $u_bv_i$, except for $O((q-1)(\smax+\tmax))$
of those choices for which $\tc ib$ or $\tc ie$ is a 2-cycle.

For choosing $j$ and edges $u_jv_c$ and $u_hv_j$,
we start with $W$ choices and subtract 
$O(\smax\tmax)$ for $j=i$
and $O(\smax\tmax(\smax+\tmax))$ that would give a forbidden edge $u_iv_j$ or $u_jv_i$, 
and $O(\smax^2+\tmax^2)$ where $c$ or $h$ is a 
previously-chosen index.
Since the actual number of choices at this point
in the analysis cannot be negative,
we can write
it as $\max\{0,W+O(\smax\tmax(\smax+\tmax))\}$.
In addition, $O((q-1)(\smax+\tmax))$ choices lie
in 2-cycles $\tc jh$ or $\tc jc$.

At this stage we divide by 2 because we could
have chosen $i$ and $j$ in the other order.
So we have that the number of choices of $\{i,j\}$
and their incident edges in $E_2$ is
$\frac12 W_1(q)W_2(q)$, where
\begin{align*}
  W_1(q) &= W+O((q-1)(\smax+\tmax)) \\
  W_2(q) &= \max\bigl\{0,W+O(\smax\tmax(\smax+\tmax))\bigr\}
          +O((q-1)(\smax+\tmax)).
\end{align*}
Now we can choose the remaining two edges.
that do not belong
to 2-cycles and do not use a previously-chosen index.  Other exclusions are:
\begin{itemize}\itemsep=0pt
\item[$u_dv_a$:] $O(\smax\tmax)$ for $u_bv_a\in G'$ or $u_dv_c\in G'$,
  $O(\smax^2)$ for $u_cv_d\in G'$, $O(\tmax^2)$ for $u_av_b\in G'$;
\item[$u_fv_g$:]  $O(\smax\tmax)$ for $u_fv_e\in G'$ or $u_hv_g\in G'$, 
  $O(\smax^2)$ for $u_ev_f\in G'$, $O(\tmax^2)$ for $u_gv_h\in G'$.
\end{itemize} 
In summary, the number of reverse switchings is
\[
   N_R = \dfrac12 W_1(q) W_2(q) \bigl(S - 2(q-1) +O(\smax^2+\tmax^2)\bigr)^2.
\]

When $\calT_{q-1}\ne\emptyset$, we can write
\[
   \abs{\calT_q}=(N_R/N_F)\abs{\calT_{q-1}}
   = \Dfrac{1}{q}\, \abs{\calT_{q-1}}
   \,\bigl(A(q) - (q-1)C(q)\bigr),
\]
where
\begin{align*}
   A(q) &= \frac{W \max\{0,W+O(\smax\tmax(\smax+\tmax))\}}{2S^2} \biggl(1+O\biggl(\frac{\smax^2+\tmax^2}{S}\biggr)\biggr) \\[0.5ex]
   C(q) &= O\biggl( \frac{(\smax+\tmax)W
    + (\smax\tmax+N_1)(\smax^2+\tmax^2)}{S^2}
    \biggr).
\end{align*}

When $\calT_{q-1}=\emptyset$, we can choose
$A(q)$ and $C(q)$ arbitrarily, so we choose
$A(q)=W^2/(2S^2)$ and $C(q)=0$.

Now we apply Lemma~\ref{summation-lemma}.
We have $A(q)\ge 0$ and $A(q)-(q-1)C(q)\ge 0$ by
their definitions.

In the value of $A(q)$, note that we always have
$\max\{0,W+O(\smax\tmax(\smax+\tmax))\}
= W + O(\smax\tmax(\smax+\tmax))$, since
$W=O(\smax\tmax(\smax+\tmax))$ if the value of
the maximum is~0.
Also, since $W\le \min\{\smax,\tmax\} S$,
we have
\[
 \frac{W^2(\smax^2+\tmax^2)}{S^3}
\le \frac{\min\{\smax,\tmax\}(\smax^2+\tmax^2)W}{S^2}
=O\biggl(\frac{\smax\tmax (\smax+\tmax)W}{S^2}\biggr).
\]
Therefore, for $q\in[N_1]$,
\[
A(q) = \frac{W^2}{2S^2}+
 O\biggl(\frac{\smax\tmax(\smax+\tmax)W}{S^2}\biggr).
\]
From the definition of $N_1$ and the theorem
assumptions, we infer that
$\abs{A(q)}/N_1\le\frac1{48} + o(1)$.
Also, $C(q)=o(1)$, so we can take $\hat c=1/(2e^3)$
in Lemma~\ref{summation-lemma}.

Next we check that $A(q)C(q)=O\bigl(\smax\tmax
(\smax+\tmax)W/S^2\bigr)$.
Since $C(q)=o(1)$, it suffices to show that
$W C(q)=O\bigl(\smax\tmax(\smax+\tmax)\bigr)$.
This follows from $\smax\tmax+N_1=o(S)$ and
$W\le \min\{\smax,\tmax\} S$.
The same bound holds for $A(q)C(q)^2$ since $C(q)=o(1)$.

Applying Lemma~\ref{summation-lemma} and
Lemma~\ref{2cycbound},
\begin{align*}
   \Dfrac{1}{\abs{\calT_0}}\sum_{q\ge 0}\,
   \abs{\calT_q} &= 
   \Dfrac{1+O(S^{-2})}{\abs{\calT_0}}\sum_{q=0}^{N_1}\,
   \abs{\calT_q} \\
   &= \exp\biggl( \frac{W^2}{2S^2} + O\biggl(
        \frac{\smax\tmax(\smax+\tmax)W}{S^2}\biggr)\biggr)
        + O(S^{-2}).
\end{align*}
Since the value of the left side is at least 1, we can
move the added term to inside the exponential, where it
is covered by the other error term.
This completes the proof.
\end{proof}

Combining Theorem~\ref{2cyclemma} with Corollary~\ref{loopfreethm} leads to the following.

\begin{cor}\label{2cycans}
 Suppose $\smax\tmax=o(S^{2/3})$, $\smax^2+\tmax^2=o(S)$,
 and $\smax\tmax(\smax+\tmax)W=o(S^2)$.
  Then the number of oriented graphs with
  degrees $(\svec,\tvec)$ is
  \[
      \frac{S!}{\prod_{i\in [n]} (s_i!\,t_i!)}
      \exp\biggl( Q(\svec,\tvec) - \frac{W}{S}
      - \frac{W^2}{2S^2}
      + O\biggl(\frac{\smax^3\tmax^3}{S^2} +
      \frac{\smax\tmax(\smax+\tmax)W}{S^2}\biggr)\biggr).
  \]
  In particular, for $d=o(n^{1/3})$, the number of regular oriented graphs
  of in-degree and out-degree $d$ is
  \[
    \frac{(dn)!}{(d!)^{2n}}
    \exp\biggl( -\frac{2d^2+1}{2} + O\biggl(\frac{d^3}{n}\biggr)\biggr).
  \]
\end{cor}

\subsection{Orientations of undirected graphs}

Let $\dvec=(d_1,\ldots,d_n)$ be a graphical degree sequence.
  Define $\dmax=\max_{i\in [n]} d_i$,
  $D=\sum_{i\in [n]} d_i$ and $D_2=\sum_{i\in [n]} (d_i)_2$.

\begin{thm}[\cite{symmetric}]\label{undirected}
  If $\dmax^4=o(D)$, the number of undirected simple graphs with degree
  sequence $\dvec$ is
  \[
      \frac{D!}{(D/2)!\,2^{D/2}\prod_{i\in [n]} d_i!}
      \exp\biggl( -\frac{D_2}{2D} - \frac{D_2^2}{4D^2}
         + O\biggl(\frac{\dmax^4}{D}\biggr)\biggr).
   \]
\end{thm}

Now let $\deltavec=(\delta_1,\ldots,\delta_n)$ be such that
$\dvec/2-\deltavec$ and $\dvec/2+\deltavec$ are non-negative integer
sequences.
Define $\varDelta_2=\sum_{i\in [n]} \delta_i^2$
and $V=\sum_{i\in [n]} \delta_i d_i$.

\begin{thm}\label{orients}
   Consider a uniformly random undirected graph with degree
   sequence~$\dvec$.
   Suppose $\dmax^4=o(D)$ and $\sum_{i\in [n]}\delta_i=0$.
   Then the expected number of orientations with in-degrees
   $\dvec/2-\deltavec$ and out-degrees $\dvec/2+\deltavec$ is
   \[
      \frac{2^{D/2}}{\binom{D}{D/2}}
      \prod_{i\in [n]}\binom{d_i}{d_i/2+\delta_i}
      \exp\biggl( -\Dfrac{3}{4} + \frac{4\varDelta_2}{D}
        - \frac{4\varDelta_2^2}{D^2} 
        + \frac{2V^2}{D^2}
        + O\biggl(\frac{\dmax^4}{D}\biggr)\biggr).
   \]
   In particular, if the entries of $\dvec$ are even, the
   expected number of Eulerian orientations is
      \[
      \frac{2^{D/2}}{\binom{D}{D/2}}
      \prod_{i\in [n]}\binom{d_i}{d_i/2}
      \exp\biggl( -\Dfrac{3}{4} 
        + O\biggl(\frac{\dmax^4}{D}\biggr)\biggr).
   \]
\end{thm}
\begin{proof}
  Let $\svec=\dvec/2-\deltavec$ and $\tvec=\dvec/2+\deltavec$.
  The expectation is the value in Corollary~\ref{2cycans} divided
  by the value in Theorem~\ref{undirected},
  using $S=\dfrac12 D$, $S_2=\dfrac14 D_2-\dfrac14 D+\varDelta_2-V$,
  $T_2=\dfrac14 D_2-\dfrac14 D+\varDelta_2+V$ and $W=\dfrac14 D_2+\dfrac14 D-\varDelta_2$.
  All the terms of $Q(\svec,\tvec)$ except the first fit into the
  error term.
  Eulerian orientations have $\deltavec=(0,\ldots,0)$,
  so set $\varDelta_2=V=0$.
\end{proof}

If $\EO(G)$ is the number of Eulerian orientations of $G$,
then $\rho(G)=\frac 1n\log \EO(G)$ is known as the
\textit{residual entropy} of~$G$
in statistical physics.
In 1935, Pauling~\cite{pauling} proposed a  heuristic estimate for $\rho(G)$
that was later proved to be a lower bound:
\[
  \hat\rho(G) = -\Dfrac{D}{2n}\log 2 
  + \Dfrac1n\sum_{j\in [n]} \log\binom{d_i}{d_i/2} 
\]
where $\dvec$ is the degree sequence of $G$.
In~\cite{IMZ} it was shown that, under the condition
$\dmax^2=o(n)$, a uniformly random undirected
graph $G$ with degree sequence $\dvec$ has
\[
    \rho(G) = \hat\rho(G) + O\Bigl(\Dfrac{\dmax^2+\log n}{n}\Bigr).
\]
Under the stronger condition
$\dmax^4=o(D)$, Theorem~\ref{orients} sharpens this to
\[
    \rho(G) = \hat\rho(G) +\Dfrac1{2n}\log\Dfrac{\pi D}{2} - \Dfrac3{4n} 
    + O\Bigl(\Dfrac{\dmax^4}{nD}\Bigr).
\]

\nicebreak


\frenchspacing

\begin{thebibliography}{10}
\itemsep=0.1ex

\bibitem{BG2016}
V. Blinovsky and C. Greenhill, Asymptotic enumeration of sparse uniform hypergraphs
with given degrees, \textit{Europ. J. Combin.}, \textbf{51} (2016) 287--296.

\bibitem{Boll}
B. Bollob\'as and B.\,D. McKay,
The number of matchings in random regular graphs and bipartite
graphs, \textit{J. Combin. Th., Ser B}, \textbf{41} (1986) 80--91.

\bibitem{GIM}
C. Greenhill, M. Isaev and B.\,D. McKay,
Subgraph counts for dense random graphs with specified degrees, \textit{Combin. Prob. Comput.}, \textbf{30} (2021) 460--497.

\bibitem{GMbip}
C.~Greenhill and B.\,D. McKay,
Random dense bipartite graphs and directed graphs with specified degrees,
\textit{Random Structures Algorithms},
\textbf{35} (2009) 222--249. 

\bibitem{GMW}
C. Greenhill, B.\,D. McKay and X. Wang,
Asymptotic enumeration of sparse 0-1 matrices with irregular row and column sums,
\textit{J. Combin. Th., Ser. A}, \textbf{113} (2006) 291--324.

\bibitem{HM23}
M. Hasheminezhad and B.\,D. McKay,
Factorisation of the complete bipartite graph into spanning semiregular factors,
\textit{Ann. Combinatorics}, \textbf{27}
(2023) 599--613. 

\bibitem{IMZ}
M. Isaev, B.\,D. McKay and R. Zhang,
Correlation between residual entropy and spanning tree
entropy of ice-type models on graphs,
\textit{Annales de l'Institut Henri Poincare D}, (2025).

\bibitem{LMT}
J.~Larkin, B.\,D.~McKay and F.~Tian,
Subgraphs in random graphs with specified degrees and forbidden edges,
Preprint, 2025. \texttt{https://arxiv.org/abs/2510.24276}.

\bibitem{LW2022}
A.~Liebenau and N.~Wormald, Asymptotic enumeration of digraphs and bipartite graphs by degree sequence, \textit{Random Structures \& Algorithms} {\bf 62} (2022), 259--286.

\bibitem{McKay1981}
B.\,D. McKay,
Subgraphs of random graphs with specified degrees, \textit{Congressus Numerantium},
\textbf{33} (1981), 213--223.

\bibitem{silver}
B.\,D. McKay,
  Asymptotics for 0-1 matrices with prescribed line sums,
  in \textit{Enumeration and Design}, Academic Press, 1984, 225--238.

\bibitem{symmetric}
B.\,D. McKay,
Asymptotics for symmetric 0-1 matrices with prescribed row sums,
\textit{Ars Combinatoria}, \textbf{19A} (1985) 15--26. 

\bibitem{ONeil}
P.\,E. O'Neil, Asymptotics and random matrices with row-sum
and column-sum restrictions,
\textit{Bull. Amer. Math. Soc.}, \textbf{75} (1969) 1276--1282.

\bibitem{pauling}
L. Pauling,
  The structure and entropy of ice and of other crystals with some
  randomness of atomic arrangement,
   \textit{J. Amer. Chem. Soc.}, \textbf{57} (1935) 2680--2684.
   
\end{thebibliography}
\end{document}